\documentclass[10pt]{article}
\usepackage{psfrag,amsmath}
\usepackage{amssymb}
\usepackage{mathrsfs}
\usepackage{amsfonts}
\usepackage{amsmath}
\usepackage{mathrsfs,color}
\usepackage{epstopdf}
\usepackage{graphicx}
\usepackage{mathrsfs,amscd,amssymb,amsthm,amsmath,url,bm}
\usepackage{indentfirst}
\usepackage{enumerate}
\usepackage{verbatim}

\makeatletter
\renewcommand{\@seccntformat}[1]{{\csname the#1\endcsname}{\normalsize .}\hspace{.5em}}
\makeatother

\def \[{\begin{equation}}
\def \]{\end{equation}}

\newtheorem{thm}{Theorem}[section]

\newtheorem{case}{Case}
\newtheorem{lem}[thm]{Lemma}
\newtheorem{cor}[thm]{Corollary}

\begin{document}

\setlength{\baselineskip}{15pt}
\begin{center}{\Large \bf Extremal trees of given degree sequence or segment sequence with\\ respect to
Steiner 3-eccentricity}
\vspace{4mm}

{\large Xin Liu\footnote{Corresponding author. \\
\hspace*{5mm}{\it Email addresses}: liuxin0622@mails.ccnu.edu.cn (X. Liu).}}\vspace{2mm}

Faculty of Mathematics and Statistics,  Central China Normal
University, Wuhan 430079, P.R. China
\end{center}

\noindent {\bf Abstract}:\ The Steiner $k$-eccentricity of a vertex in graph $G$ is the maximum Steiner distance over all $k$-subsets containing the vertex. 
Let $\mathbb{T}_n$ be the set of all $n$-vertex trees, $\mathbb{T}_{n,\Delta}$ be the set of $n$-vertex trees with given maximum degree equal to $\Delta$, $\mathbb{T}_n^k$ be the set of $n$-vertex trees with exactly $k$ vertices of maximum degree, and let $\mathbb{T}_{n,\Delta}^k$ be the set of $n$-vertex trees with exactly $k$ vertices of given maximum degree equal to $\Delta.$
In this paper, we first determine the sharp upper bound on the average Steiner 3-eccentricity of $n$-vertex trees with given degree sequence. The corresponding extremal graph is characterized. Consequently, together with majorization theory, the unique graph among $\mathbb{T}_n$ (resp. $\mathbb{T}_{n,\Delta}$, $\mathbb{T}_n^k, \mathbb{T}_{n,\Delta}^k$) having the maximum average Steiner 3-eccentricity is identified. Then we characterize the unique $n$-vertex tree with given segment sequence having the largest average Steiner 3-eccentricity. Similarly, the $n$-vertex tree with given number of segments having the largest average Steiner 3-eccentricity is determined.

\vspace{2mm} \noindent{\it Keywords:}
Steiner eccentricity;  Degree sequence; Segment sequence;  Majorization

\vspace{2mm}

\noindent{AMS subject classification:} 05C35,\ 05C12

\section{\normalsize Introduction}\setcounter{equation}{0}
We begin by describing the details about background which will contribute to the main results. Some well-known preliminaries will be given conveniently.
\subsection{\normalsize Background}\setcounter{equation}{0}
In this paper, we consider only simple, finite and undirected graphs. Let $G = (V(G), E(G))$ be a graph with vertex set $V(G)$ and edge set $E(G)$. The number of vertices $|V(G)|$ and edges $|E(G)|$ in a graph are called the \textit{order} and \textit{size} of $G,$ respectively. Then $G-v,\, G-uv$ denote the graph obtained from $G$ by deleting vertex $v\in V(G),$ or edge $uv\in E(G),$ respectively (this notation is naturally extended if more than one vertex or edge is deleted). Similarly, $G+uv$ denotes the graph obtained from $G$ by adding an edge between $u$ and $v$ which are nonadjacent in $G$. As usual, let $P_n$ and $S_n$ denote the path and star graph on $n$ vertices, respectively. A tree $T = (V, E)$ is a connected, acyclic graph.

We refer to the vertex of degree 1 as a \textit{leaf} or a \textit{pendent vertex}. If it is of degree at least 2, then it is called an \textit{internal vertex}, and if it is of degree at least 3, then it is called an \textit{branch vertex}. An edge $e$ in a graph is a pendent edge if it is incident to a pendent vertex. A path $P$ is a pendent path if one endpoint of path has degree 1 and each internal vertex of $P$ has degree 2.

For any vertex $v \in V(G)$, let $d_G(v)$ ($d(v)$ for short) denote the \textit{degree} of $v$, i.e. the number of edges incident to $v$. The \textit{degree sequence} is the sequence of degree in descending order of vertices. 
Specially, we know that if $T$ is a tree of order $n$ and $\pi=(d_1,\dots,d_n)$ is its degree sequence, then $\sum_{i=1}^{n}d_i=2(n-1).$
A \textit{segment} in a tree $T$ is a path, where each inner vertex has degree 2, and none of the ends has degree 2. A segment is pendent if it has a pendent edge. The length sequence in descending order of all segments $l=(l_1,l_2,\dots,l_m)$ in $T$ is called \textit{segment sequence}. If $T$ is of order $n$, then $\sum_{i=1}^{m}l_i=n-1$.

One of the most basic concepts in graphs is distance. If $G$ is a connected graph and $u,v$ are two vertices of $G$, then the distance $d_{G}(u, v)$ ($d(u,v)$ for short) between $u$ and $v$ is the length of a shortest path $P_{G}(u, v)$ ($P_{uv}$ for short) connecting $u$ and $v$. The \textit{eccentricity} of a vertex $u$, written $\varepsilon_{G}(u)$, is $\max_{v\in V(G)}d(u, v)$. Denote by $diam(G)$ the \textit{diameter} of $G$, which is equal to $\max_{v\in V(G)}\varepsilon(v)$ and $rad(G)$ the \textit{radius} of $G$, which is equal to $\min_{v\in V(G)}\varepsilon(v)$. We call $P_{G}(u, v)$ a \textit{diametrical path} of $G$ if $d_{G}(u, v)= diam(G)$.

We give some notations which appear in \cite{2} and \cite{1}. Let $G$ be a connected graph of order at least 2 and let $S$ be a nonempty set of vertices of $G$. Then the \textit{Steiner distance} $d_{G}(S)$ among the vertices of $S$ (or simply the distance of $S$) is the minimum size among all connected subgraphs whose vertex sets contain $S$, that is
$$d_{G}(S) = \min\{|E(T)|:\hbox{$T$ is a connected subgraph of $G$ with $S\subseteq V(T)$}\}.$$
Note that if $T$ is a connected subgraph of $G$ such that $S\subseteq V(T)$ and $|E(T)| = d(S)$, then $T$ is a tree. Such a tree has been referred to as a \textit{Steiner tree}. Further, if $S = \{u, v\}$, then $d(S) = d(u, v)$; while if $|S| = n$, then $d(S)= n-1$. If $S \neq \emptyset$, then $d(S) \geqslant 0$. Further, $d(S) = 0$ if and only if $|S| = 1$. It's also easy to see that if $T$ is a tree and $S\subseteq V(T)$, then the $S$-Steiner tree is unique.

If $k\geq2$ is an integer and $v\in V(G)$, then the \textit{Steiner k-eccentricity} $ecc_k(v,G)$ of $v$ in $G$ is the maximum Steiner distance over all k-subsets of $V(G)$ which contain $v$, that is
$$ecc_k(v,G) = \max\{d_{G}(S):v\in S\subseteq V(G),|S| = k\}.$$
Especially, $ecc_2(v,G)$ is the standard eccentricity of the vertex $v$.
The \textit{average Steiner k-eccentricity} $aecc_k(G)$ of $G$ is the mean value of all vertices' Steiner $k$-eccentricities in $G$, that is
$$aecc_k(G) = \frac{1}{|V(G)|}\sum_{v\in V(G)}ecc_k(v,G)$$
Especially, $aecc_2(G)$ is the standard average eccentricity of $G$.
A $k$-set $S\subseteq V(G)$ is a \textit{Steiner $k$-ecc $v$-set} (or \textit{$k$-ecc v-set} for short) if $v\in S$ and $d_{G}(S) = ecc_k(v,G)$; a corresponding tree that realizes $ecc_k(v,G)$ will be called a \textit{Steiner $k$-ecc $v$-tree} (or \textit{$k$-ecc $v$-tree} for short). A vertex $v$ may have more than one $k$-ecc $v$-set, and each such set may have more than one Steiner $k$-ecc $v$-tree.

It is well-known that the Steiner tree problem on general graphs is NP-hard to solve (see \cite{3,4}), however it can be solved in polynomial time on trees \cite{5}. The Steiner distance attracts much attention, which was extensively studied on trees, joins, standard graph
products, corona products, and some others, see \cite{a1,2,a3,a4,a5}. Some extremal problems on the average Steiner $k$-distance and the $k$th
Steiner Wiener index have been studied among trees (resp. complete graphs, paths, cycles, complete bipartite graphs, and others); see \cite{b1,b2}. Some work on the average Steiner distance and the Steiner Wiener index is given in \cite{c1,c2,c3,c4}. On the other hand, the Steiner (revised) Szeged index in \cite{d6}, Steiner degree distance in \cite{d2}, multi-center Wiener index in \cite{d4}, Steiner Harary index in \cite{d5}, Steiner Gutman index in \cite{d1} and Steiner hyper-Wiener index in \cite{d3} are topological indices related to the Steiner distance. For more advances on the Steiner distance, one may be referred to the nice survey paper \cite{d7}.

Very recently,  Li, Yu and Klav\v{z}ar \cite{1} first studied the average Steiner 3-eccentricity index of trees: Some general properties of the Steiner 3-eccentricity of tree are given; tree transformation which does not increase the average Steiner 3-eccentricity is presented. Consequently, some lower and/or upper bounds for the average Steiner 3-eccentricity of trees with some given parameters are established. Li and Yu \cite{11} studied the average Steiner 3-eccentricity index of block graphs: Two graph transformations are given on block graphs. Based
on these transformations, the lower and upper bounds on the average Steiner 3-eccentricity index of block graphs with a fixed block order sequence are established.

Motivated by \cite{10a,11,1}, It is natural for us to further
this study by considering the average Steiner 3-eccentricity index of trees with some given parameters. In this paper we first give a tree transformation which increases the average Steiner 3-eccentricity index of trees. We then establish the sharp upper bound on the average Steiner 3-eccentricity index of trees with given degree sequence. At last we identify the graph among $n$-vertex trees with given segment sequence (resp. with given number of segment sequence) having the smallest average Steiner 3-eccentricity index.
\subsection{\normalsize Main results}\setcounter{equation}{0}

\begin{figure}
\centering
\includegraphics[width=100mm]{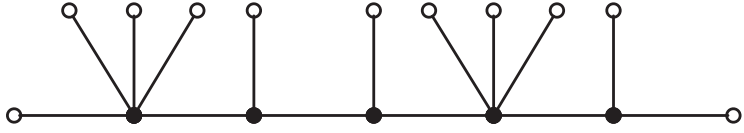} \\
    \caption{A caterpillar with given degree sequence (5,5,3,3,3,1,1,\dots,1).}
\label{fig1}
\end{figure}
A \textit{caterpillar} is a tree with the property that a path remains if all leaves are deleted. The resulting path is called a \textit{spine}. Fig.~\ref{fig1} shows a caterpillar with given degree sequence (5,5,3,3,3,1,1,\ldots,1). Let $\mathcal{T}_{n,\pi}$ be the set of $n$-vertex trees each of which has degree sequence $\pi = (d_1,\dots,d_n)$, where $n>2$. Let $\mathcal{C}_{n,\pi}^*$ be the set of $n$-vertex caterpillars each of which has degree sequence $\pi$.

Following the above notation our first main result characterizes the $n$-vertex tree having the maximum average Steiner 3-eccentricity among $\mathcal{T}_{n,\pi}$.
\begin{thm}\label{thm2.1}
Let $T$ be in $\mathcal{T}_{n,\pi},$ where $\pi = (d_1,\dots,d_n)$ with $d_1\geqslant \cdots \geqslant d_k\geqslant 2>d_{k+1}=\cdots =d_n=1.$ Then
\begin{align}\label{eq000}
\begin{split}
aecc_{3}(T) &\leqslant \left\{
                       \begin{array}{ll}
                         \frac{nk+2n-k}{n}, & \hbox{if $d_1\geqslant 3;$} \\
                         n-1, & \hbox{if $d_1 = 2$.}
                       \end{array}
                     \right.
\end{split}
\end{align}
with equality if and only if $T \in \mathcal{C}_{n,\pi}^*$.
\end{thm}

A \textit{generalized star} with segment sequence $l=(l_1,l_2,\dots,l_m)$, denoted by $S(l_1,l_2,\dots,l_m)$, is a tree with exactly one branch vertex obtained by identifying one end of each of the $m$ segments. The unique branch vertex is called \textit{center}. The generalized star $S(l_1,l_2,\dots,l_m)$ is \textit{balanced} if $l_1-l_m\leqslant 1$. We denote the unique $n$-vertex balanced generalized star with $m$ segments by $ST_{n,m}$. Let $\mathscr{T}_n^l$ denote the set of $n$-vertex trees each of which has segment sequence $l=(l_1,l_2,\dots,l_m)$, where $l_1\geqslant l_2\geqslant \cdots \geqslant l_m$. Let $\mathscr{T}_{n,m}$ be the set of $n$-vertex trees each of which has $m$ segments.

Following the above notation our next main result identifies the tree among $\mathscr{T}_n^l$ having the minimum average Steiner 3-eccentricity.
\begin{thm}\label{thm2.7}
The generalized star with segment sequence $l=(l_1,\dots,l_m)$ is the unique tree which minimizes the average Steiner 3-eccentricity in $\mathscr{T}_n^l$.
\end{thm}
Our last result characterizes the tree among $\mathscr{T}_{n,m}$ having the minimum average Steiner 3-eccentricity.
\begin{thm}\label{thm2.8}
The generalized star $ST_{n,m}$ minimizes the average Steiner 3-eccentricity in $\mathscr{T}_{n,m}$.
\end{thm}

In the rest of this section we recall some important known results. In Section 2, we introduce one transformation on trees, which makes degree sequence constant but does not decrease the average Steiner 3-eccentricity index. In section 3, we first give the proof of Theorem \ref{thm2.1}. Then we compare two different degree sequences according to their average Steiner 3-eccentricity indices if they satisfy with the majorization partial ordering of degree sequences. Consequently, we determine the sharp upper bounds on the average Steiner 3-eccentricity index of trees with given parameters. In section 4 we give the proofs of Theorem \ref{thm2.7} and Theorem \ref{thm2.8}. Some concluding remarks are given in the last section.

\subsection{\normalsize The preliminaries}\setcounter{equation}{0}
For the rest of our introduction we recall the following important preliminaries. The first one provides a method of calculating the value of $ecc_3(v,T)$ in $T$ when fixing the vertex $v$.

\begin{lem}[\cite{1}]\label{lem2.2}
Let $v$ be a vertex of a tree $T$ and let $S=\{v,x,y\}$ be a $3$-ecc $v$-set, where the $v,x$-path $P$ is a longest path in $T$ starting from $v$. Then $d_{T}(y,p)= ecc_{T}(P)$.
\end{lem}

Note that $ecc_{T}(P)=\max_{u\in V(T)}\{d_{T}(u,P)\}$. Through lemma \ref{lem2.2}, we realize that the value of $ecc_3(v,T)$ for certain vertex $v$ in tree $T$ is equal to the length of longest path $P$ starting from $v$ plus $ecc_{T}(P)$. Recall the following construction, which comes from \cite{10a}. \vspace{2mm}

\noindent {\bf $\pi$-transformation}: Let $T$ be a tree and let $P = P(u, v, T)$ be a path containing at least one edge, satisfying every internal
vertex of $P$ is of degree $2$ in $T.$ Let $X$ be the maximal subtree containing $u$ in the tree $T- E(P),$ and $Y$ be the maximal
subtree containing $v$ in the graph $T- E(P)$. Assume, without loss of generality, that $\varepsilon_{Y}(v) \geqslant \varepsilon_{X}(u).$ Then
the \textit{$\pi$-transformation} $\pi(T)$ of $T$ is defined as $T' = \pi(T) = T - \{(u, w) : w \in N_X (u)\} + \{(v, w) : w \in N_X(u)\}.$ The inverse transformation is $T = \pi^{-1}(T') = T' - \{(v, w) : w \in N_X(u)\}+  \{(u, w) : w \in N_X (u)\}$. The $\pi$-transformation is depicted in Fig.~2, which is from \cite{10a}.
\begin{figure}[h!]
  \centering
  \includegraphics[width=100mm]{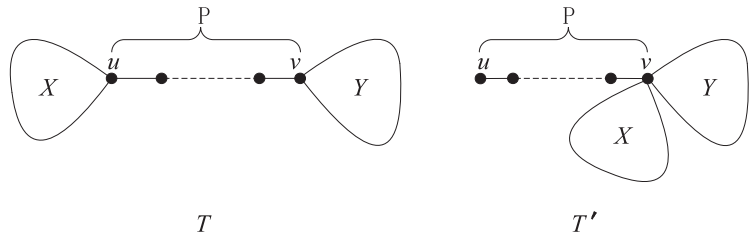}\\
  \caption{$T'=\pi(T)$ and $T=\pi^{-1}(T').$}\label{f2}
\end{figure}
\begin{lem}\label{lem1.5}
Under the notation as above, $aecc_k(T)\geqslant aecc_k(T').$
\end{lem}

Clearly, the following result is a direct consequence of Lemma \ref{lem1.5}.
\begin{cor}\label{thm6.1}
Under the notation as above, $aecc_3(T)\geqslant aecc_3(T').$
\end{cor}
\section{\normalsize A novel transformation on trees}\setcounter{equation}{0}\setcounter{case}{0}
In this section, we introduce a novel transformation on trees. Then one may describe the monotonicity of the average Steiner 3-eccentricity index with respect to the novel transformation on a tree. \vspace{2mm}

\noindent {\bf $\sigma$-transformation}: Let $T$ be a tree with degree sequence $\pi$ and let $P = v_0v_1\ldots v_k\ldots v_d$ be a diametric path of $T$. Assume that $v_ky$ is not a pendant edge of $T$. Let $X$ be the maximal subtree containing $y$ in the graph $T- v_ky$, then $\varepsilon_{X}(y)\geqslant1$ for $\vert V(X)\vert\geqslant 2$. Assume, without loss of generality, that $k \leqslant \lfloor \frac{d}{2} \rfloor.$ Then
the \textit{$\sigma$-transformation} $\sigma(T)$ of $T$ is defined as $T' = \sigma(T) = T - \{(y, w) : w \in N_X (y)\} + \{(v_d, w) : w \in N_X(y)\}.$ The $\sigma$-transformation is depicted in Fig.~3.
\begin{figure}[h!]
\centering
\psfrag{1}{$T$}
\psfrag{2}{$T'$}
\psfrag{3}{$v_0$}
\psfrag{4}{$v_1$}
\psfrag{5}{$v_k$}
\psfrag{6}{$y$}
\psfrag{9}{$v_{d-1}$}
\psfrag{A}{$v_d$}
\psfrag{e}{$X$}
\psfrag{k}{$Y$}
\includegraphics[width=150mm]{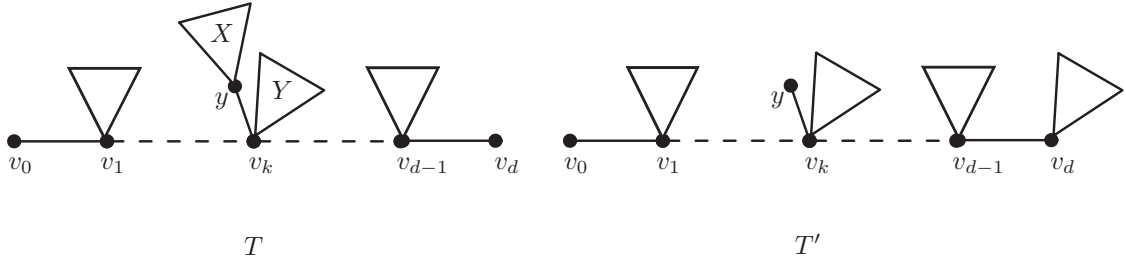} \\
    \caption{$\sigma$-transformation: $T'=\sigma(T)$.}
    \label{fig5}
\end{figure}

\begin{thm}\label{lem3.1}
Under the notation in the $\sigma$-transformation above, one has $aecc_3(T')>aecc_3(T)$.
\end{thm}
\begin{proof}
According to the $\sigma$-transformation, we know that $V(T)=V(T'),\, d_{T'}(y)=d_T(v_d)=1,\, d_{T'}(v_d)=d_T(y)$ and $d_T(u)=d_{T'}(u)$ for all $u\in V(T)\setminus \{y, v_d\}$. Hence, $T'=\sigma(T)$ and $T$ have the same degree sequence.

For $0\leqslant i\leqslant d,$ let $T_{v_i}$ (resp. $T_{v_i}'$) be the maximal subtree containing $v_i$ in $T-E(P)$ (resp. $T'-E(P)$). Then $T_{v_i}\cong T_{v_i}'$ for $i\in \{0,\dots,k-1, k+1,\ldots,d-1\}$. Put $Y:= T\left[ V\left(T_{v_k}\right)\backslash V(X)\right]$.

Firstly we show that $ecc_{3}(v,T')-ecc_{3}(v,T)\geqslant -1$ when $v=v_d$. For convenience, let $A_0=\{1,\dots,d-1\}$. Consider that $v_0$,$v_d$-path is the longest path starting from $v_d$ in $T$ and $T'$. Then by Lemma \ref{lem2.2}, we find that
\begin{align}\label{eq002}
ecc_{3}(v_d,T') &= d_{T}(v_0,v_d)+\max_{\substack{t\in A_0\setminus \{k\}}} \{\varepsilon_{T_{v_k}}(v_k),\varepsilon_{Y}(v_k),1,\varepsilon_{X}(y)\}\notag\\
&=d+\max_{\substack{t\in A_0\setminus \{k\}}} \{\varepsilon_{T_{v_k}}(v_k),\varepsilon_{Y}(v_k),1,\varepsilon_{X}(y)\}.
\end{align}
If $\varepsilon_{X}(y)+1> \max_{\substack{t\in A_0\setminus \{k\}}} \{\varepsilon_{T_{v_t}}(v_t),\varepsilon_{Y}(v_k)\}$, by Lemma \ref{lem2.2}, we find that $ecc_{3}(v_d,T)=d_{T}(v_0,v_d)+\varepsilon_{X}(y)+1=d+\varepsilon_{X}(y)+1$. From \eqref{eq002}, we obtain $ecc_{3}(v_d,T')=d+\varepsilon_{X}(y),$ which implies that $ecc_{3}(v_d,T')-ecc_{3}(v_d,T)=-1$.
If $\varepsilon_{X}(y)+1\leqslant \max_{\substack{t\in A_0\setminus \{k\}}}\{\varepsilon_{T_{v_t}}(v_t),\varepsilon_{Y}(v_k)\}$, by Lemma \ref{lem2.2}, we get that $ecc_{3}(v_d,T) =d_{T}(v_0,v_d) +\max_{t\in A_0\setminus \{k\}}\{\varepsilon_{T_{v_k}}(v_k),\varepsilon_{Y}(v_k)\} = d+\max_{t\in A_0\setminus \{k\}}\{\varepsilon_{T_{v_k}}(v_k),\varepsilon_{Y}(v_k)\}.$ Together with \eqref{eq002}, we obtain that
$$ecc_{3}(v_d,T')-ecc_{3}(v_d,T)\geqslant \max_{\substack{t\in A_0\setminus \{k\}}} \{\varepsilon_{T_{v_k}}(v_k),\varepsilon_{Y}(v_k),1 \}-\max_{t\in A_0\setminus \{k\}}\{\varepsilon_{T_{v_k}}(v_k),\varepsilon_{Y}(v_k)\}\geqslant 0.$$
We may also obtain $ecc_{3}(v_d,T')-ecc_{3}(v_d,T)\geqslant -1.$

Next we show that $ecc_{3}(v,T')-ecc_{3}(v,T)\geqslant 0$ for all $v\in V(T)\backslash\{v_d\}$ according to the following several possible cases.

{\bf Case 1.} $v = v_{j}$ for $j=0,\dots,k-1.$\ In this case, let $A_1=\{j+1,\dots,d-1\}.$ Since $j<k \leqslant \lfloor \frac{d}{2} \rfloor$, we realise that $v_j$,$v_d$-path is the longest path starting from $v_j$ in $T$. And the length of the longest path starting from $v_j$ in $T'$ is $d_{T}(v_j,v_d)+\varepsilon_{X}(y)$.

We first consider $j=d_{T}(v_0,v_j)\geqslant \max\{\varepsilon_{T_{v_t}}(v_t): t\in A_1\}$. Note that $\max\{\varepsilon_{T_{v_t}}(v_t): t\in A_1\}\geqslant \max\{\varepsilon_{T'_{v_t}}(v_t): t\in A_1\}$. Hence, by Lemma \ref{lem2.2}, we obtain that $ecc_{3}(v_{j},T)=d_{T}(v_j,v_d)+d_{T}(v_0,v_j)=d$ and $ecc_{3}(v_{j},T')=d_{T}(v_j,v_d)+\varepsilon_{X}(y)+d_{T}(v_0,v_j)=d+\varepsilon_{X}(y)$. Therefore, $ecc_{3}(v_j,T)< ecc_{3}(v_j,T')$ holds in this subcase.

Now we consider $j=d_{T}(v_0,v_j)< \max\{\varepsilon_{T_{v_t}}(v_t): t\in A_1\}$. By Lemma \ref{lem2.2}, we get $ecc_{3}(v_{j},T')=d_{T}(v_j,v_d)+\varepsilon_X(y)+d_{T}(v_0,v_j)$ if $d_{T}(v_0,v_j)\geqslant \max\{\varepsilon_{T'_{v_t}}(v_t): t\in A_1\}$, and $ecc_{3}(v_{j},T')= d_{T}(v_j,v_d)+\varepsilon_X(y)+\max_{ t\in A_1\setminus \{k\}}$ $\{\varepsilon_{T'_{v_t}}(v_t), \varepsilon_Y(v_k),1\}$ otherwise. So, we could summarize as
\begin{align}\label{eq:2.01}
\begin{split}
ecc_{3}(v_{j},T')&=\left\{
                       \begin{array}{ll}
                         d+\varepsilon_X(y), & \hbox{if $j \geqslant \max\{\varepsilon_{T'_{v_t}}(v_t): t\in A_1\};$} \\
                         (d-j)+\varepsilon_X(y)+\max_{ t\in A_1\setminus \{k\}}\{\varepsilon_{T'_{v_t}}(v_t),\varepsilon_Y(v_k),1\}, & \hbox{otherwise.}
                       \end{array}
                     \right.
\end{split}
\end{align}
If $\varepsilon_{X}(y)+1> \max_{\substack{t\in A_1\setminus \{k\}}} \{\varepsilon_{T_{v_t}}(v_t),\varepsilon_{Y}(v_k)\}$, by Lemma \ref{lem2.2},
$ecc_{3}(v_j,T)=d_{T}(v_j,v_d)+\varepsilon_X(y)+1=(d-j)+\varepsilon_X(y)+1$. Together with \eqref{eq:2.01}, we obtain
$$
ecc_{3}(v_j,T')-ecc_{3}(v_j,T)\geqslant\max_{\substack{t\in A_1\setminus \{k\}}} \{ \varepsilon_{T'_{v_t}}(v_t), \varepsilon_Y(v_k), 1\}-1
\geqslant 0.
$$
If $\varepsilon_{X}(y)+1\leqslant \max_{\substack{t\in A_1\setminus \{k\}}} \{\varepsilon_{T_{v_t}}(v_t),\varepsilon_{Y}(v_k)\}$, then
$ecc_{3}(v_{j},T)=d_{T}(v_j,v_d)+\max_{t\in A_1\setminus \{k\}}\{\varepsilon_{T_{v_t}}(v_t),\varepsilon_{Y}(v_k)\}=(d-j)+\max_{t\in A_1\setminus \{k\}}\{\varepsilon_{T_{v_t}}(v_t),\varepsilon_{Y}(v_k)\}.$
Together with \eqref{eq:2.01}, one has
$$ecc_{3}(v_j,T')-ecc_{3}(v_j,T)\geqslant\varepsilon_X(y)+\max_{ t\in A_1\setminus \{k\}}\{\varepsilon_{T'_{v_t}}(v_t),\varepsilon_Y(v_k),1\}-\max_{t\in A_1\setminus \{k\}}\{\varepsilon_{T_{v_t}}(v_t),\varepsilon_{Y}(v_k)\}> 0.$$
Thus, our result holds in this case.

{\bf Case 2.}\ $v\in V(T_{v_j})$ with $j\in \{0,\dots,k-1\}.$\ \ In this case, one has
$$
ecc_{3}(v,T)=d_T(v,v_j)+ecc_{3}(v_j,T),\ \ \ \ \ \ ecc_{3}(v,T')=d_T(v,v_j)+ecc_{3}(v_j,T').
$$
This gives
$$ecc_{3}(v,T')-ecc_{3}(v,T)=ecc_{3}(v_j,T')-ecc_{3}(v_j,T).$$
In view of Case 1, we obtain that $ecc_{3}(v,T')-ecc_{3}(v,T)\geqslant 0$ holds in this case.

{\bf Case 3.}\  $v=v_k.$\ \  In this case, let $A_3=\{k+1,\dots,d-1\}$. Since $k \leqslant \lfloor \frac{d}{2} \rfloor$, we realise that $v_k$,$v_d$-path is the longest path starting from $v_k$ in $T$. And the length of the longest path starting from $v_k$ in $T'$ is $d_{T}(v_k,v_d)+\varepsilon_{X}(y)$.

We first consider $k=d_{T}(v_0,v_k)\geqslant \max\{\varepsilon_{T_{v_t}}(v_t): t\in A_3\}$. Note that $\max\{\varepsilon_{T_{v_t}}(v_t): t\in A_3\}\geqslant \max\{\varepsilon_{T'_{v_t}}(v_t): t\in A_3\}$. Hence, by Lemma \ref{lem2.2}, we obtain that $ecc_{3}(v_{k},T)=d_{T}(v_k,v_d)+d_{T}(v_0,v_k)=d$ and $ecc_{3}(v_{k},T')=d_{T}(v_k,v_d)+\varepsilon_{X}(y)+d_{T}(v_0,v_k)=d+\varepsilon_{X}(y)$. Therefore, $ecc_{3}(v_k,T)< ecc_{3}(v_k,T')$ holds for this subcase.

Now we consider $k=d_{T}(v_0,v_k)< \max\{\varepsilon_{T_{v_t}}(v_t): t\in A_3\}$. According to Lemma \ref{lem2.2}, by similar calculation we get that
\begin{align*}
\begin{split}
ecc_{3}(v_k,T)&=(d-k)+\max\{\varepsilon_{T_{v_t}}(v_t): t\in A_3\}.
\end{split}\\
\begin{split}
ecc_{3}(v_k,T')&=\left\{
                    \begin{array}{ll}
                        d+\varepsilon_{X}(y), & \hbox{if $k \geqslant \max\{\varepsilon_{T'_{v_t}}(v_t): t\in A_3\};$} \\
                        (d-k)+\varepsilon_{X}(y)+\max\{\varepsilon_{T'_{v_t}}(v_t): t\in A_3\}, &\hbox{otherwise.}
                    \end{array}
                    \right.
\end{split}
\end{align*}
This gives $ecc_{3}(v_k,T')-ecc_{3}(v_k,T)\geqslant\varepsilon_{X}(y)> 0$. Therefore, our result holds in this case.

{\bf Case 4.}\   $v\in \{y\}\cup (V(Y)\backslash\{v_k\}).$ In this case, we may get
$$
ecc_{3}(v,T)=d_T(v,v_k)+ecc_{3}(v_k,T),\ \ \ \ \ \ ecc_{3}(v,T')=d_T(v,v_k)+ecc_{3}(v_k,T').
$$
This gives $$ecc_{3}(v,T')-ecc_{3}(v,T)=ecc_{3}(v_k,T')-ecc_{3}(v_k,T).$$
In view of Case 3, we may obtain $ecc_{3}(v,T')-ecc_{3}(v,T)>0$ consequently, as desired.

{\bf Case 5.}\    $v\in V(X)\backslash\{y\}.$ In this case, let $A_5=\{k+1,\dots,d-1\}$ and $B_5=\{1,\dots,d-1\}$. Consider that $v$,$v_0$-path is the longest path starting from $v$ in $T'$. Then by Lemma \ref{lem2.2}, we get
\begin{align}\label{eq003}
 ecc_{3}(v,T')&=d_{T'}(v,v_0)+\max_{\substack{t\in B_5\setminus \{k\}}} \{\varepsilon_{T_{v_t}}(v_t), \varepsilon_{Y}(v_k),1\}\notag\\
 &=d+d_T(v,y)+\max_{\substack{t\in B_5\setminus \{k\}}} \{\varepsilon_{T_{v_t}}(v_t), \varepsilon_{Y}(v_k),1\}
\end{align}
Since $k \leqslant \lfloor \frac{d}{2} \rfloor$, we realise that $v$,$v_d$-path is the longest path starting from $v$ in $T$. If $k=d_{T}(v_0,v_k)\geqslant \max\{\varepsilon_{T_{v_t}}(v_t): t\in A_5\}$, by Lemma \ref{lem2.2}, $ecc_{3}(v,T)=d_{T}(v,v_d)+d_{T}(v_0,v_k)=d+d_T(v,y)+1$. Together with \eqref{eq003}, we obtain that
$$
 ecc_{3}(v,T')-ecc_{3}(v,T)=\max_{\substack{t\in B_5\setminus \{k\}}} \{\varepsilon_{T_{v_t}}(v_t), \epsilon_{Y}(v_k),1\}-1\geqslant 0.
$$
If $k=d_{T}(v_0,v_k)< \max\{\varepsilon_{T_{v_t}}(v_t): t\in A_5\}$, by Lemma \ref{lem2.2}, $ecc_{3}(v,T)=d_{T}(v,v_d)+\max\{\varepsilon_{T_{v_t}}(v_t): t\in A_5\}=(d-k)+d_T(v,y)+1+\max\{\varepsilon_{T_{v_t}}(v_t): t\in A_5\}$. Note that $A_5\subseteq B_5\setminus \{k\}$, then $\max_{\substack{t\in B_5\setminus \{k\}}} \{\varepsilon_{T_{v_t}}(v_t), \epsilon_{Y}(v_k),1\}\geqslant \max\{\varepsilon_{T_{v_t}}(v_t): t\in A_5\}$. Hence, together with \eqref{eq003},
$$
 ecc_{3}(v,T')-ecc_{3}(v,T)=k+\max_{\substack{t\in B_5\setminus \{k\}}} \{\varepsilon_{T_{v_t}}(v_t), \epsilon_{Y}(v_k),1\}-\max\{\varepsilon_{T_{v_t}}(v_t): t\in A_5\}-1\geqslant 0.
$$
Therefore, our result holds in this case.

{\bf Case 6.}\   $v = v_j$ for $j=k+1,\dots,d-1.$\ \ In this case, let $A_6=\{1,\dots,j-1\}$ and $B_6=\{j+1,\dots,d-1\}$.

We first consider $j<\lfloor \frac{d}{2} \rfloor$. Then, we realise that $v_j$,$v_d$-path is the longest path starting from $v_j$ in $T$. Hence, we obtain by Lemma \ref{lem2.2} that
\begin{align}
\label{eq:1}
ecc_{3}(v_{j},T)&=d_{T}(v_j,v_d)+\max_{t\in B_6}\{\varepsilon_{T_{v_t}}(v_t),d_{T}(v_0,v_j)\}\notag\\
&=(d-j)+\max_{t\in B_6}\{\varepsilon_{T_{v_t}}(v_t), j \}
\end{align}
Meanwhile, note that $\lfloor \frac{d}{2} \rfloor < \lfloor \frac{d+\varepsilon_{X}(y)}{2} \rfloor$. We say the length of the longest path starting from $v_j$ in $T'$ is $d_{T}(v_j,v_d)+\varepsilon_{X}(y)$. By Lemma \ref{lem2.2},
\begin{align}
\label{eq:2}
ecc_{3}(v_{j},T')&=d_{T}(v_j,v_d)+\varepsilon_{X}(y)+\max_{t\in B_6}\{\varepsilon_{T_{v_t}}(v_t), d_{T}(v_0,v_j)\}\notag\\
&=(d-j)+\varepsilon_{X}(y)+\max_{t\in B_6}\{\varepsilon_{T_{v_t}}(v_t), j\}
\end{align}
Combine \eqref{eq:1} with \eqref{eq:2}, it gives $ecc_{3}(v_j,T')-ecc_{3}(v_j,T)=\varepsilon_{X}(y)>0$.

Now we consider $j\geqslant \lfloor \frac{d}{2} \rfloor$. Since the length of the longest path starting from $v_j$ in $T'$ is $d_{T}(v_0,v_j)$ if $j\geqslant \lfloor \frac{d+\varepsilon_{X}(y)}{2} \rfloor$ and $d_{T}(v_j,v_d)+\varepsilon_{X}(y)$ otherwise. By Lemma \ref{lem2.2}, we could obtain the following equation through similar calculation above,
\begin{align}\label{eq:2.02}
\begin{split}
ecc_{3}(v_j,T')&=\left\{
                    \begin{array}{ll}
                    j+\max_{\substack{t\in A_6\setminus \{k\}}} \{ \varepsilon_{T_{v_t}}(v_t),\varepsilon_{Y}(v_k),1,(d-j)+\varepsilon_{X}(y)\}, &\hbox{if $j\geqslant \lfloor \frac{d+\varepsilon_{X}(y)}{2} \rfloor;$}\\ (d-j)+\varepsilon_{X}(y)+\max_{t\in B_6}\{\varepsilon_{T_{v_t}}(v_t),j\}, &\hbox{otherwise}.
                    \end{array}
                    \right.
\end{split}
\end{align}
Note that $j\geqslant \lfloor \frac{d}{2} \rfloor$, then $v_0$,$v_j$-path is the longest path starting from $v_j$ in $T$. If $\varepsilon_{X}(y)+1> \max_{\substack{t\in A_6\setminus \{k\}}}$ $\{\varepsilon_{T_{v_t}}(v_t),\varepsilon_{Y}(v_k)\}$, then by Lemma \ref{lem2.2}, $ecc_{3}(v_j,T)=d_{T}(v_0,v_j)+\varepsilon_{X}(y)+1=j+\varepsilon_{X}(y)+1$. Together with \eqref{eq:2.02}, since $j\leqslant d-1$, one has
$$ecc_{3}(v_j,T')-ecc_{3}(v_j,T)\geqslant d+\varepsilon_{X}(y)-(j+\varepsilon_{X}(y)+1) \geqslant 0.$$
If $\varepsilon_{X}(y)+1\leqslant \max_{\substack{t\in A_6\setminus \{k\}}} \{\varepsilon_{T_{v_t}}(v_t),\varepsilon_{Y}(v_k)\}$, then
$ecc_{3}(v_j,T)=d_{T}(v_0,v_j)+\max_{t\in A_6\setminus \{k\}}\{\varepsilon_{T_{v_t}}(v_t),d_{T}(v_j,v_d)\}
=j+\max_{t\in A_6\setminus \{k\}}\{\varepsilon_{T_{v_t}}(v_t),d-j\}$.
Together with \eqref{eq:2.02}, we could also verify that $ecc_{3}(v_j,T')-ecc_{3}(v_j,T)\geqslant 0 $.
In fact, if $ecc_{3}(v_j,T)=d$, then $$ecc_{3}(v_j,T')-ecc_{3}(v_j,T)\geqslant d+\varepsilon_{X}(y)-d>0;$$ if $ecc_{3}(v_j,T)=j+\max\{\varepsilon_{T_{v_t}}(v_t): t\in A_6\setminus \{k\}\}$ and $j\geqslant \lfloor \frac{d+\varepsilon_{X}(y)}{2} \rfloor$, then
$$ecc_{3}(v_j,T')-ecc_{3}(v_j,T)\geqslant \max_{\substack{t\in A_6\setminus \{k\}}} \{ \varepsilon_{T_{v_t}}(v_t),\varepsilon_{T'_y}(v_k),1\}-\max\{\varepsilon_{T_{v_t}}(v_t): t\in A_6\setminus \{k\}\}\geqslant 0;$$
if $ecc_{3}(v_j,T)=j+\max\{\varepsilon_{T_{v_t}}(v_t): t\in A_6\setminus \{k\}\}$ and $j< \lfloor \frac{d+\varepsilon_{X}(y)}{2} \rfloor$, since $\max\{\varepsilon_{T_{v_t}}(v_t): t\in A_6\setminus \{k\}\}\leqslant t\leqslant j< \lfloor \frac{d+\varepsilon_{X}(y)}{2} \rfloor< d-j+\varepsilon_{X}(y)$, then
$$ecc_{3}(v_j,T')-ecc_{3}(v_j,T)\geqslant d+\varepsilon_{X}(y)-(j+\max\{\varepsilon_{T_{v_t}}(v_t): t\in A_6\setminus \{k\}\})>0.$$
Above all, this gives $ecc_{3}(v_j,T')-ecc_{3}(v_j,T)\geqslant 0 $. Thus our result holds in this case.

{\bf Case 7.}\ $v\in V(T_{v_j})$ with $j=k+1,\dots,d-1.$ \ In this case, it is easy to see that
$$
ecc_{3}(v,T)=d_T(v,v_j)+ecc_{3}(v_j,T),\ \ \ \ \ \ ecc_{3}(v,T')=d_T(v,v_j)+ecc_{3}(v_j,T').
$$
This gives us
$$ecc_{3}(v,T')-ecc_{3}(v,T)=ecc_{3}(v_j,T')-ecc_{3}(v_j,T).$$
In view of Case 6, we obtain that $ecc_{3}(v,T')-ecc_{3}(v,T)\geqslant 0,$ as desired.

All the situations have been discussed. We realise that the strict inequality $ecc_{3}(v,T')-ecc_{3}(v,T)\geqslant\varepsilon_{X}(y)> 0$ holds for at least two vertices $v_k$ and $y$ stated in Case 3 and 4. And $ecc_{3}(v,T')-ecc_{3}(v,T)\geqslant 0$ holds in other cases. By the definition of average Steiner 3-eccentricity, it thus follows that
\begin{align*}
aecc_3(T')-aecc_3(T)&=\sum_{v\in V(T)}\left(ecc_{3}(v,T')-ecc_{3}(v,T)\right)\\
&\geqslant \sum_{v\in \{v_k,v_d,y\}}\left(ecc_{3}(v,T')-ecc_{3}(v,T)\right)\\
&\geqslant 2\varepsilon_{X}(y)-1\\
&>0
\end{align*}
which makes Theorem \ref{lem3.1} hold.
\end{proof}

\section{\normalsize Proofs of Theorems \ref{thm2.1} and its direct consequences}\setcounter{equation}{0}
\begin{proof}[\bf Proof of Theorem \ref{thm2.1}]
Let $T_0$ be an arbitrary tree of order $n$ with given degree sequence $\pi=(d_1,\dots,d_n)$. We consider the following cases:\setcounter{case}{0}
\begin{case}\label{case0001}
$T_0$ is a caterpillar.
\end{case}
If $d_1=2$, then we have $T_0\cong P_n$ and $aecc_3(T_0)=n-1$. Otherwise, $d_1\geq3$, then $T_0$ could be shown as Fig~\ref{fig4}.
\begin{figure}
\centering
\psfrag{1}{$v_1$}
\psfrag{2}{$v_2$}
\psfrag{3}{$v_{k-1}$}
\psfrag{4}{$v_{k}$}
\psfrag{5}{$\dots$}
\psfrag{6}{$y_1$}
\psfrag{7}{$y_2-1$}
\psfrag{8}{$y_{k-1}-1$}
\psfrag{9}{$y_k$}
\includegraphics[width=80mm]{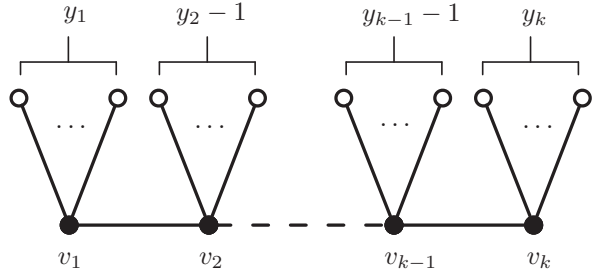} \\
\caption{Caterpillar $T_0$ with given degree sequence $(d_{v_1},\dots,d_{v_k},d_{v_{k+1}},\dots,d_{v_n})$.}
\label{fig4}
\end{figure}
Let $d(v_i)=y_i+1\geqslant 2 $ for $i=1,\dots,k$ and $d(v_{k+1})=\ldots =d(v_n)=1$, then all the vertices in spine have same values of $ecc_{3}(v,T_0)=k+1 $ and all the leaves have same values of $ecc_{3}(v,T_0)=k+2$.
\begin{align*}
aecc_3(T_0)=\frac{(k+2)(n-k)+(k+1)k}{n}=\frac{(n-1)k+2n}{n}
\end{align*}
the equality holds in (\ref{eq000}).
\begin{case} $T_0$ is not a caterpillar.\end{case}
In this case, we can repeatedly apply $\sigma$-transformation to $T_0$ until no further $\sigma$-transformation is possible. Without loss of generality, we assume that the process ends in $t>1$ steps and denote $T_t$ as the final tree. Now we claim that we must necessarily arrive at a caterpillar $T_t$. Suppose on the contrary that $T_t$ is not a caterpillar. In this case, there is at least one internal vertex which is adjacent to certain vertex of the longest path $P_1$ in $T_t$. Let $u$ be such internal vertex and $v$ be one of leaves of $P_1$ which is farther from $u$. Now one can do a $\sigma$-transformation to obtain a new tree defined as $T_{t+1}=T_t-\{ux:x\in N_{T_t}(u)\backslash V(P_1)\}+\{vx:x\in N_{T_t}(u)\backslash V(P_1)\}$. This contradicts the fact that $T_t$ is achieved in the last step. Hence, in the last step we must arrive at a caterpillar. From Theorem \ref{lem3.1} and Case \ref{case0001} above, we have
\begin{align*}
\begin{split}
aecc_3(T_0)<aecc_3(T_t)&= \left\{
                       \begin{array}{ll}
                         \frac{nk+2n-k}{n}, & \hbox{if $d_1\geq3;$} \\
                         n-1, & \hbox{if $d_1 = 2$.}
                       \end{array}
                     \right.
\end{split}
\end{align*}
Thus the inequality in (\ref{eq000}) is strict.

As a matter of fact, caterpillars with given degree sequence $\pi$ has the maximal average Steiner 3-eccentricity which implies that $\mathcal{C}_{n,\pi}^*$ is the set of extremal trees.
\end{proof}
From the proof of Theorem \ref{thm2.1}, we also realise that the value of average Steiner 3-eccentricity of a caterpillar of order $n$ is only relevant to the length of its spine.

Before we derive the maximum average Steiner 3-eccentricity in some classes of trees, we recall the notion of majorization\cite{6}.
Let $x=(x_1,\dots,x_n)$ and $y=(y_1,\dots,y_n)$ be two non-increasing and non-negative sequences. We say that $x$ \textit{majorizes} $y$, written as $x\succeq y$, if, for $k = 1,\dots,n-1,$
\begin{align*}
\textit{$\sum_{i=1}^{k}x_i \geqslant \sum_{i=1}^{k}y_i$ \ \ and \ \ $\sum_{i=1}^{n}x_i = \sum_{i=1}^{n}y_i.$}
\end{align*}
We will apply this concept to degree sequence on trees.

Now we give a kind of partition of $\mathbb{T}_n$ defined as followed. Let $\mathbb{T}_n=\{T: \hbox{ $T$ is a tree of order $n$} \}$ and $\Pi=\{\pi: \hbox{$\pi$ is degree sequence of any tree in $\mathbb{T}_n$}\}$. Note that
\begin{enumerate}[(i)]
\item $\mathbb{T}_n = \bigcup_{\pi \in \Pi}\mathcal{T}_{n,\pi};$
\item $\mathcal{T}_{n,\pi_1}\cap \mathcal{T}_{n,\pi_2} = \emptyset$ for any two degree sequences $\pi_1,\pi_2 \in \Pi$.
\end{enumerate}
Then, we say that $\{\mathcal{T}_{n,\pi}:\pi \in \Pi\}$ is a partition of $\mathbb{T}_n$.

For convenience, we denotes $T_\pi^*$ as any caterpillar of $\mathcal{T}_{n,\pi}$.
\begin{thm}\label{thm4.1}
For two degree sequences $\pi_1$ and $\pi_2$ in $\Pi$, where both the maximal degrees of those are not less than 3. If $\pi_1\succeq \pi_2$ ,then $aecc_{3}(T_{\pi_1}^*)\leqslant aecc_{3}(T_{\pi_2}^*)$.
\end{thm}
\begin{proof}
Let $\pi_1=(d_1,\dots,d_n)$ and $\pi_2=(d_1',\dots,d_n')$ be two degree sequences in $\Pi$ with $d_1\geqslant3$, $d_1'\geqslant3$.
Assume that $d_s\geqslant2$, $d_{s+1}=\dots=d_n=1$ and $d_t'\geqslant2$, $d_{t+1}'=\dots=d_n'=1$ where $1<s,t<n$.

We claim that if $\pi_1\succeq \pi_2$, then $s\leqslant t$. In fact, suppose that $s>t$, then $d_i=d_i'=1$ for $i=s+1,\dots,n$. Considering that $\sum_{i=1}^{n}d_i=\sum_{i=1}^{n}d_i'$=2(n-1), we get $\sum_{i=1}^{s}d_i=\sum_{i=1}^{s}d_i'$. Since $\sum_{i=1}^{s}d_i'=(s-t)+\sum_{i=1}^{t}d_i'=1+\sum_{i=1}^{s-1}d_i' $, we get
$$
\sum_{i=1}^{s-1}d_i =\sum_{i=1}^{s}d_i - d_{s}=\sum_{i=1}^{s}d_i' - d_{s}
=\left((s-t)+\sum_{i=1}^{t}d_i' \right)- d_{s}
<(s-t)+\sum_{i=1}^{t}d_i' - 1
=\sum_{i=1}^{s-1}d_i'
$$
which contradicts $\pi_1\succeq \pi_2$.

From Theorem \ref{thm2.1}, since $d_1\geqslant3$ and $d_1'\geqslant3$, we obtain that
$$aecc_{3}(T_{\pi_1}^*) = \frac{(n-1)s+2n}{n}\leqslant \frac{(n-1)t+2n}{n}=aecc_{3}(T_{\pi_2}^*),$$
as desired.
\end{proof}

Using Theorems \ref{thm2.1} and \ref{thm4.1}, we may deduce extremal graphs with the maximum average Steiner 3-eccentricity in some class of graphs. The next corollary establishes the sharp upper bound on the largest average Steiner 3-eccentricity of trees among $\mathbb{T}_n$.
\begin{cor}\label{thm2.2}
Let $T$ be an $n$-vertex tree. Then $aecc_{3}(T)\leqslant n-1$ with equality if and only if $T\cong P_n$.
\end{cor}
\begin{proof}
Clearly, $T\cong P_3$ for $n=3$. By a direct calculation, we get $aecc_{3}(P_3)=2$. The equality holds obviously.

Now we assume $n>3$. Let $\hat{\pi}=(\underbrace{2,\dots,2}_{n-2},1,1)$ and $\bar{\pi}=(3,\underbrace{2,\dots,2}_{n-4},1,1,1)$.
Let $T$ be any $n$-vertex tree and $\pi$ be the degree sequence of $T$. If $\pi=\hat{\pi}$, then $T\cong P_n$ and $aecc_3(T)=n-1$ which makes the equality hold. Otherwise, $\pi=(d_1,\dots,d_n) \in \Pi\backslash\{\hat{\pi}\}$ and $d_1\geqslant 3$. It implies $T\in\mathbb{T}_n\backslash P_n$. We claim that $d_{n-2}=d_{n-1}=d_n= 1$. In fact, if not, then $d_{n-2}>1$ which results $\pi=\hat{\pi}$. It thus makes a contradictory.

We will verify that $\bar{\pi}\preceq \pi$. Since $d_1\geqslant 3$ and $d_{n-2}=d_{n-1}=d_n= 1$, then $\sum_{i=1}^{n-3}d_i=2(n-1)-3=3+2(n-4)$. There must exist $t\in\{2,\dots n-3\}$ so that $d_t\geqslant 2$ and $d_{t+1}=\dots=d_{n}=1$. If $t=n-3$, then $\pi=\bar{\pi}$ which makes $\bar{\pi}\preceq \pi$ hold. If $t\leqslant n-4$, we have
\begin{align*}
d_1&\geqslant 3;\\
\sum_{i=1}^{m}d_i&=d_1+\sum_{i=2}^{m}d_i\geqslant 3+2(m-1),\ \hbox{for $m=2,\dots,t$;}\\
\sum_{i=1}^{m}d_i&=\sum_{i=1}^{n-3}d_i-\sum_{i=m+1}^{n-3}d_i=3+2(n-4)-(n-3-m)>3+2(m-1) ,\ \hbox{for $m=t+1,\dots,n-4$;}\\
\sum_{i=1}^{m}d_i&=3+2(n-4)+(m-n+3) ,\ \hbox{for $m=n-3,\dots,n$.}
\end{align*}
Hence, $\bar{\pi}\preceq \pi$.
Then by Theorem \ref{thm2.1} and \ref{thm4.1}, we obtain
$$aecc_{3}(T)\leqslant aecc_{3}(T_\pi^*)\leqslant aecc_{3}(T_{\bar{\pi}}^*)=\frac{(n-1)(n-3)+2n}{n}$$
Since $aecc_{3}(P_n)-aecc_{3}(T_{\bar{\pi}}^*)=n-1-\frac{(n-1)(n-3)+2n}{n}=\frac{n-3}{n}>0$ for $n>3$. It shows that $aecc_{3}(T)< aecc_{3}(P_n)$ for all $T\in \mathbb{T}_n\backslash P_n$. So the equality holds if and only if $T\cong P_n$.

Above all, we complete the proof.
\end{proof}

The next four results characterize extremal trees having the maximum average Steiner 3-eccentricity with some given parameters.

Let $\mathbb{T}_{n,\Delta}$ be the set of trees of order $n>2$ with maximum degree $\Delta\geq3$ and $\mathcal{B}_{n,\Delta}$ be the set of trees, each of which is obtained from $P_{n-\Delta+2}$ by attaching $\Delta-2$ pendent vertices to certain internal vertex of $P_{n-\Delta+2}$.
\begin{cor}\label{thm2.3}
Let $T$ be in $\mathbb{T}_{n,\Delta}$. Then $aecc_{3}(T)\leqslant \frac{(n-1)(n-\Delta)+2n}{n}$
with equality if and only if $T\in \mathcal{B}_{n,\Delta}$.
\end{cor}
\begin{proof}
Let $\Pi_{\Delta}=\{\pi \in \Pi:d_1=\Delta\}$ and $\pi_{\Delta}=(\Delta,\underbrace{2,\dots,2}_{n-\Delta-1},\underbrace{1,\dots,1}_{\Delta})$. Clearly, $\pi_{\Delta}\in \Pi_{\Delta}$ and $\{\mathcal{T}_{n,\pi}:\pi \in \Pi_{\Delta}\}$ is a partition of $\mathbb{T}_{n,\Delta}$.

If $T\in \mathcal{B}_{n,\Delta}$, then $T$ is a caterpillar with degree sequence $\pi_{\Delta}$. According to Theorem \ref{thm2.1}, $aecc_{3}(T)=\frac{(n-1)(n-\Delta)+2n}{n}$
which makes the equality hold.

Otherwise, let $T$ be any tree in $\mathbb{T}_{n,\Delta}\backslash\mathcal{B}_{n,\Delta}$ and $\pi$ be the degree sequence of $T$.
Using similar proof above, we could obtain that $\pi_{\Delta}\preceq \pi$.
By Theorem \ref{thm4.1} and \ref{thm2.1}, we get
$$aecc_{3}(T)\leqslant aecc_{3}(T_\pi^*)\leqslant aecc_{3}(T_{\pi_{\Delta}}^*)=\frac{(n-1)(n-\Delta)+2n}{n}.$$
If $\pi\neq \pi_{\Delta}$, then $T\in\mathbb{T}_{n,\Delta}\backslash\mathcal{T}_{n,\pi_{\Delta}}$. Comparing two degree sequences, we find that the number of degrees greater than $1$ in $\pi_{\Delta}$ is strictly larger than that in $\pi$. From the proof of Theorem \ref{thm4.1}, the inequality $aecc_{3}(T_\pi^*) < aecc_{3}(T_{\pi_{\Delta}}^*)$ holds. If $\pi= \pi_{\Delta}$, then $T\in \mathcal{T}_{n,\pi_{\Delta}}\backslash\mathcal{B}_{n,\Delta}$. By Theorem \ref{thm2.1}, the inequality $aecc_{3}(T)<aecc_{3}(T_{\pi}^*)$ holds.

We complete our proof.
\end{proof}

\begin{figure}
\centering
\psfrag{1}{$\lceil\frac{k}{2}\rceil$}
\psfrag{2}{$n-2k-2$}
\psfrag{3}{$\lfloor\frac{k}{2}\rfloor$}
\includegraphics[width=100mm]{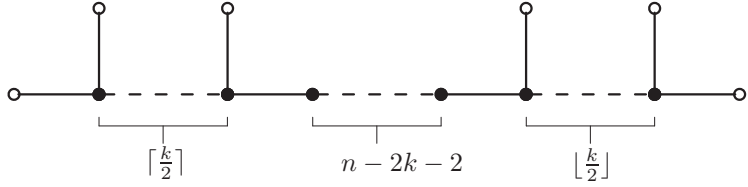} \\
\caption{A caterpillar in $\mathcal{C}_n^{k}$}
\label{fig2}
\end{figure}
Given two integers $k$ and $n$ satisfying $n>2$ and $1\leqslant k \leqslant n-3$, let $\mathbb{T}_n^k$ be the set of $n$-vertex trees, each of which contains exactly $k$ vertices of maximum degree. Let $\mathcal{C}_n^{k}$ be the set of $n$-vertex caterpillars with degree sequence $\pi_k=(\underbrace{3,\dots,3}_k,\underbrace{2,\dots,2}_{n-2k-2},1,\dots,1)$. The graph depicted in Fig.~\ref{fig2} is one of caterpillars in $\mathcal{C}_n^{k}$.
\begin{cor}\label{thm2.4}
Let $T$ be in $\mathbb{T}_n^k$. Then
$aecc_{3}(T)\leqslant \frac{(n-1)(n-k-2)+2n}{n}$ with equality if and only if $T$ is in $\mathcal{C}_n^{k}$.
\end{cor}
\begin{proof}
If $T \in \mathcal{C}_n^{k}$, then $T$ is a caterpillar whose degree sequence is $\pi_k$. From Theorem \ref{thm2.1}, $aecc_3(T)=\frac{(n-1)(n-k-2)+2n}{n}$. The equality holds in this case.

Otherwise, let $T$ be any tree in $\mathbb{T}_n^k\backslash\mathcal{C}_n^{k}$ and $\pi=(d_1,\dots,d_n)$ be the degree sequence of $T$. Without loss of generality, we assume $\Delta$ as the maximum degree of $T$. Then $\Delta\geqslant 3$ for $k \leqslant n-3$. Since there are exactly $k$ vertices of maximum degree, we get that $d_1=\ldots =d_k=\Delta$. We could obtain by similar proof that $\pi_k\preceq \pi$. By Theorem \ref{thm2.1} and \ref{thm4.1},
$$aecc_{3}(T)\leqslant aecc_{3}(T_\pi^*)\leqslant aecc_{3}(T_{\pi_k}^*)=\frac{(n-1)(n-k-2)+2n}{n} .$$
If $\pi\neq \pi_k$, then $T\in\mathbb{T}_n^k\backslash\mathcal{T}_{n,\pi_k}$. We find that the number of degrees greater than $1$ in $\pi_k$ is strictly larger than that in $\pi$. From the proof of Theorem \ref{thm4.1}, the strict inequality $aecc_{3}(T_\pi^*) < aecc_{3}(T_{\pi_k}^*)$ holds. If $\pi =\pi_k$, then $T\in\mathcal{T}_{n,\pi_k}\backslash\mathcal{C}_n^{k}$. By Theorem \ref{thm2.1}, $aecc_{3}(T)<aecc_{3}(T_{\pi}^*)$. So the inequality holds in this case.

As a whole, we complete the proof.
\end{proof}

\begin{figure}
\centering
\psfrag{1}{$\lceil\frac{k}{2}\rceil$}
\psfrag{2}{$n-(\Delta-1)k-2$}
\psfrag{3}{$\lfloor\frac{k}{2}\rfloor$}
\psfrag{4}{$\dots$}
\psfrag{5}{$\Delta-2$}
\includegraphics[width=100mm]{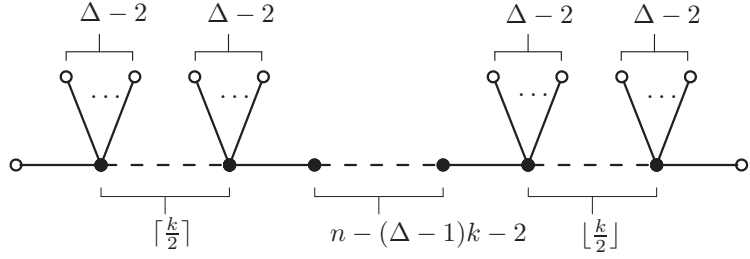} \\
\caption{A caterpillar in $\mathcal{C}_{n,\Delta}^{k}$}
\label{fig3}
\end{figure}
Given three integers $k, n$ and $\Delta$ satisfying $n>3, k\geq1, \Delta>2$ and $n\geqslant 2+k(\Delta-1)$, let $\mathbb{T}_{n,\Delta}^k$ be the set of trees of order $n$ with exactly $k$ vertices of maximum degree equal to $\Delta$, and let $\mathcal{C}_{n,\Delta}^{k}$ be the set of caterpillars with degree sequence $\pi_{k,\Delta}=(\underbrace{\Delta,\dots,\Delta}_k,\underbrace{2,\dots,2}_{n-(\Delta-1)k-2},1,\dots,1)$. The graph depicted in Fig.~\ref{fig3} is a caterpillar in $\mathcal{C}_{n,\Delta}^{k}$.
\begin{cor}\label{thm2.5}
Let $T$ be in $\mathbb{T}_{n,\Delta}^k$. Then $aecc_{3}(T)\leqslant \frac{(n-1)(n-(\Delta-2)k-2)+2n}{n}$ with equality if and only if $T\in \mathcal{C}_{n,\Delta}^{k}$.
\end{cor}
\begin{proof}
If $T\in \mathcal{C}_{n,\Delta}^{k}$, then from Theorem \ref{thm2.1} we know that $T$ is a caterpillar with degree sequence $\pi_{k,\Delta}$ and $aecc_3(T)=\frac{(n-1)(n-(\Delta-2)k-2)+2n}{n}$. The equality holds in this case.

Otherwise, let $T$ be any tree in $\mathbb{T}_{n,\Delta}^k\backslash\mathcal{C}_{n,\Delta}^{k}$ and $\pi=(d_1,\dots,d_n)$ be the degree sequence of $T$. As there are exactly $k$ vertices of given maximum degree equal to $\Delta$ in $T$, then $d_1=\ldots =d_k=\Delta$ and $\Delta-1\geqslant d_i\geqslant 1$ for $i=k+1,\dots,n$. Using similar proof, we obtain that $\pi_{k,\Delta}\preceq \pi$. By Theorem \ref{thm2.1} and Theorem \ref{thm4.1},
$$aecc_{3}(T)\leqslant aecc_{3}(T_\pi^*)\leqslant aecc_{3}(T_{\pi_{k,\Delta}}^*)=\frac{(n-1)(n-(\Delta-2)k-2)+2n}{n}.$$
If $\pi\neq \pi_{k,\Delta}$, then $T\in \mathbb{T}_{n,\Delta}^k\backslash \mathcal{T}_{n,\pi_{k,\Delta}}$. We find that the number of degrees greater than $1$ in $\pi_{k,\Delta}$ is strictly larger than that in $\pi$. By the proof of Theorem \ref{thm4.1}, it makes $aecc_{3}(T_\pi^*)< aecc_{3}(T_{\pi_{k,\Delta}}^*)$ hold.
If $\pi =\pi_{k,\Delta}$, then $T \in \mathcal{T}_{n,\pi_{k,\Delta}} \backslash \mathcal{C}_{n,\Delta}^{k}$. By Theorem \ref{thm2.1}, $aecc_{3}(T)<aecc_{3}(T_{\pi}^*)$. So the inequality holds in this case.

Now, we complete the proof.
\end{proof}

\begin{cor}\label{cor4.2}
For given positive integers $n,k,\Delta$ which satisfy $n>3,k\geqslant 1, \Delta>3$ and $n\geqslant n-2-k(\Delta-2)$, then
$$aecc_{3}(T_{\pi_{k,\Delta-1}}^*)> aecc_{3}(T_{\pi_{k,\Delta}}^*).$$
\end{cor}
\begin{proof}
In the view of \ref{thm2.5}, note that
$$\pi_{k,\Delta}=(\Delta,\dots,\Delta,2,\dots,2,1,\dots,1),\ \ \ \ \ \pi_{k,\Delta-1}=(\Delta-1,\dots,\Delta-1,2,\dots,2,1,\dots,1)
$$
are degree sequences of the caterpillars $T_{\pi_{k,\Delta}^*}$ and $T_{\pi_{k,\Delta-1}}^*$ which maximize the average Steiner 3-eccentricity among all trees having exactly $k$ vertices of maximum degree equal to $\Delta$ and equal to $\Delta-1$, respectively. Since $\pi_{k,\Delta-1}\preceq \pi_{k,\Delta}$, by Theorem \ref{thm4.1}, $aecc_{3}(T_{\pi_{k,\Delta-1}}^*)\geqslant aecc_{3}(T_{\pi_{k,\Delta}}^*).$

Besides, we find that the number of degrees greater than $1$ in $\pi_{k,\Delta-1}$ is strictly larger than that in $\pi_{k,\Delta}$. From the proof of Theorem \ref{thm4.1}, $aecc_{3}(T_{\pi_{k,\Delta-1}}^*)> aecc_{3}(T_{\pi_{k,\Delta}}^*).$
Hence the assertion holds.
\end{proof}
\section{\normalsize Proofs of Theorems \ref{thm2.7} and \ref{thm2.8}}\setcounter{equation}{0}\setcounter{case}{0}
\begin{proof}[\bf{Proof of Theorem \ref{thm2.7}}]
Let $T^0$ be the tree of order $n$ with given segment sequence $l=(l_1,\dots,l_m)$ having the minimal average Steiner 3-eccentricity. If $T^0$ is a generated star, then the theorem holds.

Otherwise, assume that $T^0$ is not a generated star and apply with $\pi$-transformation mentioned in Section 2. Now we proceed according to the following steps: Since $T^0$ is not a generalized star, there must be at least two branch vertices. Choose two branch vertices $u,w$ of $V(T^0)$ so that $P_{uw}$ is a segment. Let $T_u$ (resp.$T_w$) be the connected component of $T^0-E(P_{uw})$ containing $u$ (resp.$w$). Without loss of generality, we assume that $\varepsilon_{T_u}(u)\geq\varepsilon_{T_w}(w)$. Then, let $T^1=T^0-\{wv|v\in N_{T_w}(w)\}+\{uv|v\in N_{T_w}(w)\}$. By Corollary \ref{thm6.1}, we have $aecc_3(T^0)\geqslant aecc_3(T^{1})$.
\begin{figure}
\centering
\psfrag{1}{$a$}
\psfrag{2}{$b$}
\psfrag{3}{$\vdots$}
\psfrag{4}{$v_l$}
\psfrag{5}{$T^{t-1}$}
\psfrag{6}{$T^{t}$}
\includegraphics[width=100mm]{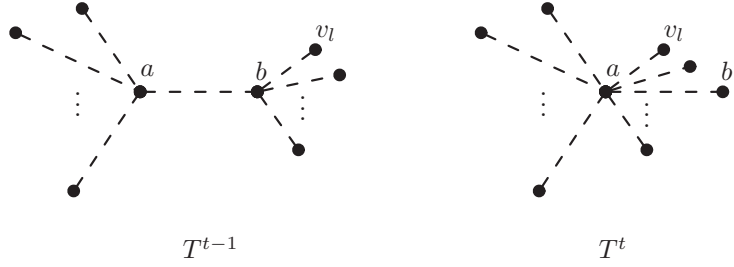} \\
\caption{$T^{t-1}$ and $T^{t}$ in the last but one step}
\label{fig10}
\end{figure}
We could repeatedly apply $\pi$-transformation satisfying the requirements above (choose two branch vertices satisfying the path between them is a segment) on $T^0$ until no further transformation is possible. Without loss of generality, we assume that the process ends in $t>1$ steps and denote $T^i$, for $i=1,\dots,t$, as the tree obtained after the $i_{th}$ transformation. We claim that in the last step we must arrive at a generated star. Suppose on the contrary, let $T^t$ be the tree achieved in the last step and $T^t$ is not a generated star. Then there must be two branch vertices $r,s$ of $V(T^t)$ so that $P_{rs}$ is a segment of $T^t$. Let $T_r$ (resp.$T_s$) be the connected component of $T^t-E(P_{rs})$ containing $r$ (resp.$s$). Without loss of generality, we assume that $\varepsilon_{T_r}(r)\geq\varepsilon_{T_s}(s)$. Now one can do a $\pi$-transformation to obtain a new tree defined as $T^{t+1}=T^{t}-\{sv|v\in N_{T_s}(s)\}+\{rv|v\in N_{T_s}(s)\}$. This contradicts the fact that $T^t$ is the final tree which achieved in the last step.
From Corollary \ref{thm6.1}, we get
$$aecc_3(T^0)\geqslant aecc_3(T^1)\geqslant \dots \geqslant aecc_3(T^{t-1})\geqslant aecc_3(T^t).$$

Note that, in the last but one step, there must be exactly two branch vertices in $T^{t-1}$, written as $a$ and $b$. In fact, if there exists at least three branch vertices in $T^{t-1}$, we need more than one step to transform it into a generated star. That is a contradiction. Similarly, we denote $T_a$ and $T_b$ as the connected components of $a$ and $b$ in $T^{t-1}-E(P_{ab})$. Assume that $\varepsilon_{T_a}(a)\geq\varepsilon_{T_b}(b)$ (shown in Fig~\ref{fig10}). We find that both $T_a$ and $T_b$ are generated stars. We can verify that there must exist a leaf $v_l\in V(T_b)$ in $T^{t-1}$ so that $ecc_3(v_l,T^{t-1})> ecc_3(v_l,T^t)$. Since the length of the longest path starting from $v_l$ in $T^{t-1}$ is strictly larger than that in $T^t$, by Lemma \ref{lem2.2}, it holds.
Besides, according to the proof of Lemma \ref{lem1.5} in \cite{10a}, it gives that the Steiner k-eccentricity of every vertex $v\in V(T^{t-1})$ is not less than that in $T^t$. This assertion deservedly holds when $k=3$. So $ecc_3(v,T^{t-1})\geqslant ecc_3(v,T^t)$ holds for all vertices. Then $aecc_3(T^{t-1})> aecc_3(T^t)$ which implies the inequality holds in this case.

As a result, we say that generalized star with given segment sequence $l$ is the unique tree which has the minimum average Steiner 3-eccentricity among all trees in $\mathscr{T}_n^l$.
\end{proof}

\begin{proof}[\bf{Proof of Theorem \ref{thm2.8}}]
In view of Theorem \ref{thm2.7}, the tree in $\mathscr{T}_{l}$ minimizing the average Steiner 3-eccentricity must be of the form $S(l_1,l_2,\dots,l_m)$ where $(l_1,l_2,\dots,l_m)$ is a segment sequence. Clearly, the tree having $m$ segments will be a generalized star if it minimizes the average Steiner 3-eccentricity.

Let $T$ be the generalized star having the largest index and $u$ be the unique branch vertex called center. Assume that $T\ncong ST_{n,m}$ which implies $l_1-l_m\geqslant 2$. Let $P_i$ be the segment whose length is $l_i$, $i=1,2,\dots,m$ and $P_1=uu_1\ldots u_{l_1}, P_m=uw_1\ldots w_{l_m}.$
Let $T'=T-u_{l_1-1}u_{l_1}+w_{l_m}u_{l_1}$ (shown in Fig~\ref{fig11}). Clearly, $T'$ is a tree of order $n$ having $m$ segments. Let $A=\{3,\dots,m\}$. We will focus on the following possible cases to verify that $aecc_{3}(T)\geq aecc_{3}(T')$.

\begin{case}\label{case111} $|P_2|=l_1.$ \end{case}
\begin{enumerate}[(i)]
\item If $v=u_i$, for $i=1,\dots,l_1-1$. Then the length of the longest path starting from $v$ in both $T$ and $T'$ is $l_1+d_{T}(v,u)$. By lemma \ref{lem2.2}, $ecc_{3}(v,T)=l_1+d_{T}(v,u)+\max_{j\in A}\{l_j, l_1-i\}=l_1+i+\max_{j\in A}\{l_j, l_1-i\}$ and $ecc_{3}(v,T')=l_1+d_{T}(v,u)+\max_{j\in A\backslash\{m\}}\{l_j, l_m+1,l_1-i-1\}=l_1+i+\max_{j\in A\backslash\{m\}}\{l_j, l_m+1,l_1-i-1\}$ which implies $ecc_{3}(v,T)\geq ecc_{3}(v,T')$ holds in this subcase.
\item If $v=u_{l_1}$. Then the length of the longest path starting from $v$ in $T$ is $2l_1$ and the length in $T'$ is $l_1+l_m+1$. By lemma \ref{lem2.2}, $ecc_{3}(v,T)=2l_1+\max_{j\in A}l_j$ and $ecc_{3}(v,T')=l_1+l_m+1+\max_{j\in A\backslash\{m\}}\{l_1-1,l_j\}$ which implies $ecc_{3}(v,T)\geq ecc_{3}(v,T')$ holds in this subcase.
\item If $v\in V(P_2)\backslash \{u\}$. The length of longest path starting from $v$ in $T$ is $d_T(u,v)+l_1$. By lemma \ref{lem2.2}, $ecc_{3}(v,T)=\max_{j\in A}\{2l_1, l_1+l_j+d_T(u,v)\}$ and $ecc_{3}(v,T')=\max\{2l_1-1,d_T(u,v)+l_1+\max_{j\in A\backslash\{m\}}\{l_j-1\}, d_T(u,v)+l_1+l_m, d_T(u,v)+\max_{j\in A\backslash\{m\}}\{l_j\}+\max_{k\in A\backslash\{j,m\}}\{l_k\}\}$. By comparison, $ecc_{3}(v,T)\geq ecc_{3}(v,T')$ also holds in this subcase.
\item If $v\in V(P_i)\backslash \{u\}$ with $i=3,\dots,m-1$. Then, the length of longest path starting from $v$ is $d_T(u,v)+l_1$ in both $T$ and $T'$. By lemma \ref{lem2.2}, $ecc_{3}(v,T)=d_T(u,v)+2l_1$ and $ecc_{3}(v,T')=d_T(u,v)+l_1+\max_{j\in A\backslash\{m\}}\{l_j,l_1-1\}$. It shows that $ecc_{3}(v,T)\geq ecc_{3}(v,T')$ in this subcase.
\item If $v\in V(P_m)\backslash \{u\}.$ The length of longest path starting from $v$ is $d_T(u,v)+l_1$ in both $T$ and $T'$. Through calculation in lemma \ref{lem2.2}, we get $ecc_{3}(v,T)=d_T(u,v)+2l_1$ and $ecc_{3}(v,T')=d_T(u,v)+l_1+\max_{j\in A\backslash\{m\}}\{l_j, l_1-1\}$ which implies $ecc_{3}(v,T)\geq ecc_{3}(v,T')$ holds in this subcase.
\item If $v= u$. Then, the length of longest path starting from $v$ is $l_1$ in both $T$ and $T'$. By lemma \ref{lem2.2}, $ecc_{3}(v,T)=2l_1$ and $ecc_{3}(v,T')=l_1+\max_{j\in A\backslash\{m\}}\{l_j,l_1-1\}$. We obtain that $ecc_{3}(v,T)\geq ecc_{3}(v,T')$ in this subcase.
\end{enumerate}
In sum, $aecc_3(T)-aecc_3(T')=\sum_{v\in V(T)}[ecc_{3}(v,T)-ecc_{3}(v,T')]\geqslant 0$. The assertion holds in Case \ref{case111}.

\begin{case}\label{case112} $l_m < l_2 < l_1.$ \end{case}
\begin{enumerate}[(i)]
\item If $v=u_{l_1}$. The length of longest path starting from $v$ in $T$ is $l_1+l_2$, and the length in $T'$ is $l_1+l_m$. By lemma \ref{lem2.2}, $ecc_{3}(v,T)=l_1+l_2+\max_{j\in A}l_j$ and $ecc_{3}(v,T')=l_1+l_2+l_m$ which makes $ecc_{3}(v,T)\geqslant ecc_{3}(v,T')$ holds in this subcase.
\item If $v\in V(P_2)\backslash \{u\}$. The length of longest path starting from $v$ in $T$ is $d_T(u,v)+l_1$. By lemma \ref{lem2.2}, $ecc_{3}(v,T)=\max_{j\in A}\{l_1+l_2,l_1+l_j+d_T(u,v)\}$ and $ecc_{3}(v,T')=\max_{j\in A\backslash\{m\}}\{l_1+l_2-1,l_1-1+l_j+d_T(u,v),l_1+l_m+d_T(u,v)\}$ which makes $ecc_{3}(v,T)\geqslant ecc_{3}(v,T')$ holds in this subcase.
\item If $v\in V(P_i)\backslash \{u\}$ with $i\in\{3,\dots,m\}$. The length of longest path starting from $v$ in $T$ is $d_T(u,v)+l_1$ and the length in $T'$ is $d_T(u,v)+l_1-1$. By lemma \ref{lem2.2}, $ecc_{3}(v,T)=d_T(u,v)+l_1+l_2$ and $ecc_{3}(v,T')=d_T(u,v)+l_1+l_2-1$. By comparison, we know that $ecc_{3}(v,T)=ecc_{3}(v,T')+1$. So $ecc_{3}(v,T)\geq ecc_{3}(v,T')$ also holds in this subcase.
\item If $v=u$. The length of longest path starting from $v$ in $T$ is $l_1$ and the length in $T'$ is $l_1-1$. Then by lemma \ref{lem2.2}, $ecc_{3}(v,T)=l_1+l_2$ and $ecc_{3}(v,T')=l_1+l_2-1$ which implies $ecc_{3}(v,T)=ecc_{3}(v,T')+1$. So $ecc_{3}(v,T)\geq ecc_{3}(v,T')$ also holds in this subcase.
\item If $v\in V(P_1)\backslash \{u,u_{l_1}\}$, then the length of longest path starting from $v$ is $d_T(u,v)+l_2$ in both $T$ and $T'$. By lemma \ref{lem2.2}, we obtain
    \begin{align}
    ecc_{3}(v,T)&=\max_{j\in A}\{l_2+l_j+i, l_1+l_2\}.\label{eq1121}\\
    ecc_{3}(v,T')&=\max_{j\in A\backslash\{m\}}\{l_1+l_2-1,l_2+l_j+i,l_2+l_m+i+1\}.\label{eq1122}
    \end{align}\end{enumerate}
\begin{figure}
\centering
\psfrag{1}{$u$}
\psfrag{2}{$u_1$}
\psfrag{3}{$u_{l_1-1}$}
\psfrag{4}{$u_{l_1}$}
\psfrag{5}{$w_1$}
\psfrag{6}{$w_{l_m}$}
\psfrag{8}{$T$}
\psfrag{9}{$T'$}
\includegraphics[width=75mm]{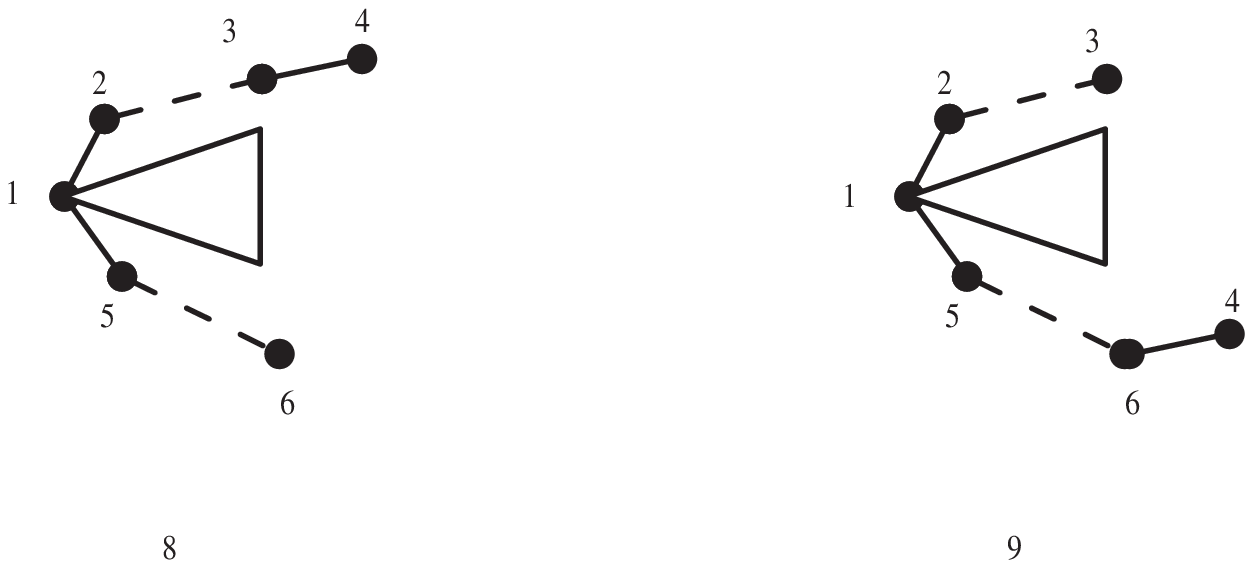} \\
\caption{$T$ and $T'$}
\label{fig11}
\end{figure}
For situation (\romannumeral5), if $m>3$ and $l_3>l_m$, then $l_3=\max_{j\in A}l_j$. By the equation (\ref{eq1121}) and (\ref{eq1122}), we get
\begin{align*}
\begin{split}
ecc_{3}(u_i,T)=\left\{
                \begin{array}{ll}
                l_1+l_2, &\hbox{if $i\leqslant l_1-l_3;$}\\
                l_2+l_3+i, &\hbox{otherwise.}
                \end{array}
                \right.
\end{split}\\
\begin{split}
ecc_{3}(u_i,T')=\left\{
                \begin{array}{ll}
                l_1+l_2-1, &\hbox{if $i\leqslant l_1-l_3-1;$}\\
                l_2+l_3+i, &\hbox{otherwise.}
                \end{array}
                \right.
\end{split}
\end{align*}
It implies that $ecc_{3}(u_i,T)=ecc_{3}(u_i,T')+1$ for $1\leqslant i\leqslant l_1-l_3-1$ and $ecc_{3}(u_i,T)=ecc_{3}(u_i,T')-1$ for $l_1-l_3 \leqslant i\leqslant l_1-1$. Since $ecc_{3}(v,T)\geqslant ecc_{3}(v,T')$ in other situations. Together with situation (\romannumeral3), we get
\begin{align*}
aecc_3(T)-aecc_3(T')&=\sum_{v\in V(T)}[ecc_{3}(v,T)-ecc_{3}(v,T')]\\
&\geqslant \sum_{v\in V(P_1\cup P_3)\backslash\{u,u_{l_1}\}}[ecc_{3}(v,T)-ecc_{3}(v,T')]\\
&\geqslant l_3-l_3\\
&= 0
\end{align*}
If $m>3$ and $l_3=l_m$ or $m=3$, then $l_m=\max_{j\in A}l_j$. By the equation (\ref{eq1121}) and (\ref{eq1122}), we get
\begin{align*}
\begin{split}
ecc_{3}(u_i,T)&=\left\{
                \begin{array}{ll}
                l_1+l_2, &\hbox{if $i\leqslant l_1-l_m;$}\\
                l_2+l_m+i, &\hbox{otherwise.}
                \end{array}
                \right.
\end{split}\\
\begin{split}
ecc_{3}(u_i,T')&=\left\{
                \begin{array}{ll}
                l_1+l_2-1, &\hbox{if $i\leqslant l_1-l_m-1;$} \\
                l_2+l_m+i+1, &\hbox{otherwise.}
                \end{array}
                \right.
\end{split}
\end{align*}
It implies $ecc_{3}(u_i,T)=ecc_{3}(u_i,T')+1$ for $1\leqslant i\leqslant l_1-l_m-1$ and $ecc_{3}(u_i,T)=ecc_{3}(u_i,T')-1$ for $l_1-l_m \leqslant i\leqslant l_1-1$. Together with situation (\romannumeral3), we get
\begin{align*}
aecc_3(T)-aecc_3(T')&=\sum_{v\in V(T)}[ecc_{3}(v,T)-ecc_{3}(v,T')]\\
&\geqslant \sum_{v\in V(P_1\cup P_m)\backslash\{u,u_{l_1}\}}[ecc_{3}(v,T)-ecc_{3}(v,T')]\\
                                            &\geqslant l_m-l_m\\
                                            &=0
\end{align*}
Then, $aecc_{3}(v,T)\geq aecc_{3}(v,T')$ holds in Case \ref{case112}.

\begin{case}\label{case113} $l_m=l_2<l_1.$ \end{case}
\begin{enumerate}[(i)]
\item If $v=u_{l_1}$. The length of longest path starting from $v$ is $l_1+l_2$ in both $T$ and $T'$. Then by lemma \ref{lem2.2}, $ecc_{3}(v,T)=l_1+2l_2=ecc_{3}(v,T')$.
\item If $v\in V(P_i)\backslash \{u\}$ with $i\in\{2,\dots,m-1\}$. The length of longest path starting from $v$ in $T$ is $d_T(u,v)+l_1$. Then by lemma \ref{lem2.2}, $ecc_{3}(v,T)=l_1+l_2+d_T(u,v)=ecc_{3}(v,T')$.
\item If $v\in V(P_m)\backslash \{u\}$. The length of longest path starting from $v$ in $T$ is $d_T(u,v)+l_1$ and the length in $T'$ is $d_T(u,v)+l_1-1$. By lemma \ref{lem2.2}, $ecc_{3}(v,T)=l_1+l_2+d_T(u,v)$ and $ecc_{3}(v,T')=l_1+l_2+d_T(u,v)-1$. We know that $ecc_{3}(v,T)=ecc_{3}(v,T')+1$ which implies $ecc_{3}(v,T)\geqslant ecc_{3}(v,T')$ holds.
\item If $v=u$. The length of longest path starting from $v$ in $T$ is $l_1$. Then by lemma \ref{lem2.2}, $ecc_{3}(v,T)=l_1+l_2=ecc_{3}(v,T')$.
\item If $v\in V(P_1)\backslash \{u,u_{l_1}\}$, then the length of longest path starting from $v$ in $T$ is $d_T(u,v)+l_2$ and the length in $T'$ is $d_T(u,v)+l_2+1$. By lemma \ref{lem2.2},
    \begin{align}
    ecc_{3}(v,T)&=\max\{l_1+l_2,2l_2+i\}.\label{eq1131}\\
    ecc_{3}(v,T')&=\max\{l_1+l_2,2l_2+i+1\}.\label{eq1132}
    \end{align}\end{enumerate}
 For situation (\romannumeral5). By the equation (\ref{eq1131}) and (\ref{eq1132}), we obtain that
\begin{align*}
\begin{split}
ecc_{3}(u_i,T)&=\left\{
                \begin{array}{ll}
                l_1+l_2, \hbox{if $i\leqslant l_1-l_2$;}\\
                2l_2+i,  \hbox{otherwise.}
                 \end{array}
                \right.
\end{split}\\
\begin{split}
ecc_{3}(u_i,T')&=\left\{
                \begin{array}{ll}
                l_1+l_2, \hbox{if $i\leqslant l_1-l_2-1$;}\\
                2l_2+i+1, \hbox{otherwise.}
                 \end{array}
                \right.
\end{split}
\end{align*}
It implies that $ecc_{3}(u_i,T)=ecc_{3}(u_i,T')$ for $1\leqslant i\leqslant l_1-l_2-1$ and $ecc_{3}(u_i,T)=ecc_{3}(u_i,T')-1$ for $l_1-l_2 \leqslant i\leqslant l_1-1$. Together with situation (\romannumeral3), we get
\begin{align*}
aecc_3(T)-aecc_3(T')&=\sum_{v\in V(T)}[ecc_{3}(v,T)-ecc_{3}(v,T')]\\
&= \sum_{v\in V(P_1\cup P_m)\backslash\{u\}}[ecc_{3}(v,T)-ecc_{3}(v,T')]\\
&=l_m-l_2\\
&=0
\end{align*}
which implies that $aecc_{3}(v,T)=aecc_{3}(v,T')$ holds in Case \ref{case113}.

As a result, $aecc_{3}(v,T)\geqslant aecc_{3}(v,T')$ holds for all cases. Since any tree $T\ncong ST_{n,m}$ which has $m$ segments could be transformed into a generalized star by the method above, meanwhile, the value of average Steiner 3-eccentricity will not increase. Then, we know that the generalized star has the minimal average Steiner 3-eccentricity in $\mathscr{T}_{n,m}$. We complete the proof of Theorem \ref{thm2.8}.
\end{proof}

\section{\normalsize Concluding remarks}
We considered the extremal problems with respect to the average Steiner 3-eccentricity among trees with given degree sequence and segment sequence. A lot of work still has to be done on the average Steiner k-eccentricity for $k\geq3$.

\end{document}